\documentclass[a4paper,11pt]{article}

\usepackage[latin1]{inputenc}
\usepackage{amsfonts,amsmath,amssymb,mathrsfs}
\usepackage{dsfont}
\usepackage{graphicx,ifthen}
\usepackage{epsfig}
\usepackage{color}
\usepackage[english]{babel}
\usepackage{latexsym}
\usepackage{theorem}
\usepackage{array}
\usepackage{natbib}
\usepackage{xspace}

\usepackage{cancel} 

{\theorembodyfont{\itshape}\theoremheaderfont{\scshape}
\newtheorem{theorem}{Theorem}
\newtheorem{lemma}{Lemma}
\newtheorem{definition}{Definition}
}
{\theorembodyfont{\upshape} \theoremheaderfont{\scshape}

}

\newcommand{\fine}{\hfill $\Box$}         

\newenvironment{proof}{
\begin{trivlist}
\item[\hspace{\labelsep}{\sc\noindent\textsc{Proof.} }]
}{\fine\end{trivlist}}

\newenvironment{proof of}[1]{
\begin{trivlist}
\item[\hspace{\labelsep}{\sc\noindent\textsc{Proof of #1.}}]
}{\fine\end{trivlist}}

\newcommand{\mle}{{\hat\theta}}
\newcommand{\edr}{\mathrm{e}}

\newcommand{\Prob}{{\rm P}}
\newcommand{\E}{{\rm E}}
\newcommand{\RR}{{\mathbb R}}
\newcommand{\equitailed}{equal tailed\xspace}

\newcommand{\half}{\Small{\frac{1}{2}}}

\newcommand{\Small}[1]{\textstyle #1 \displaystyle}
\newcommand{\comillas}[1]{``\,#1\,''}
\definecolor{iblue}{rgb}{0.1,0,0.75}
\definecolor{ired}{rgb}{0.9,0,0.1}

\newcommand{\argmin}{\operatornamewithlimits{arg\,min}}

\title{\bf Confidence distributions from likelihoods by median bias correction \protect}

\author{\Large{Pierpaolo De Blasi}\footnote{Supported by the European Research Council (ERC) through St
G "N-BNP" 306406.}\\
University of Torino 
and Collegio Carlo Alberto
\and \Large{Tore Schweder}\\
Department of Economics, University of Oslo, Norway }
\date{\today}

\graphicspath{ {graphics/} }

\begin{document}

\maketitle

\begin{abstract}
\noindent By the modified directed likelihood, higher order accurate confidence limits for a scalar parameter are obtained from the likelihood. They are conveniently described in terms of a confidence distribution, that is a sample dependent distribution function on the parameter space. In this paper we explore a different route to accurate confidence limits via tail--symmetric confidence curves, that is curves that describe \equitailed intervals at any level. Instead of modifying the directed likelihood, we consider inversion of the log-likelihood ratio when evaluated at the median of the maximum likelihood estimator. This is shown to provide \equitailed intervals, and thus an exact confidence distribution, to the third-order of approximation in regular one-dimensional models. Median bias correction also provides an alternative approximation to the modified directed likelihood which holds up to the second order in exponential families.

\end{abstract}


\section{Introduction}\label{section:1}

The level of reported confidence intervals are most often  95\%, with equal probability of missing the target at both sides. Sometimes other levels are used, but  rarely are several intervals at their different levels reported in applied work. Instead of only reporting one confidence interval we suggest to report a family of nested confidence intervals for parameters of primary interest. The family is  indexed by the confidence level $1-\alpha$ for $\alpha \in (0,\,1)$ and is  conveniently represented by what is called a {\it confidence curve}, a quantity introduced by \citet{Bir61} to give a complete picture of the estimation uncertainty.
As an example, take  $\hat \theta \sim \rm{N}(\theta, \,\sigma^2)$ for $\sigma$ known. It yields the curve $cc(\theta)=|1-2\Phi((\theta-\hat \theta)/\sigma)|$ for $\Phi(z)$ the cumulative distribution function of a $\rm{N}(0,1)$. This is a confidence curve since, for all $\alpha \in (0,\,1)$, $\{\theta:cc(\theta) \leq 1-\alpha\}=(\hat \theta+\sigma \Phi^{-1}(\alpha/2),\,\hat \theta+\sigma \Phi^{-1}(1-\alpha/2))$ is the respective confidence interval of level $1-\alpha$. In the example the confidence curve has its minimum at $\hat \theta$ which is a point estimate of $\theta$. The normal confidence curve is tail-symmetric, i.e. the probability of missing the parameter to the left equals that to the right and is $\alpha/2$ at level $1-\alpha$. A tail-symmetric confidence curve represents uniquely a confidence distribution, that is confidence curves that describe upper confidence limits. 
Confidence distribution is a term coined by \citet{Cox58} and formally defined in \citet{Sch:Hjo:02}. For scalar parameters the fiducial distributions developed by \citet{Fis30} are confidence distributions. \citet{Ney34} saw that the fiducial distribution leads to confidence intervals. \citet{Cox13} sees confidence distributions as ``simple and interpretable summaries of what can reasonably be learned from the data (and an assumed model)''. Confidence distributions are reviewed  by \citet{Xie11}, and more broadly and with more emphasis on confidence curves by \citet{Sch16}.
In location models and other simple models the confidence distribution is obtained from pivots, e.g. the normal pivot $(\hat \theta-\theta)/\sigma$ in the above example. A canonical pivot is $C(\theta)=C(\theta;\mle)=1-G(\hat\theta,\theta)$ where $G(y,\theta)=\Prob(\hat\theta\leq y;\theta)$ is the distribution function of the maximum likelihood estimator $\hat\theta$, assumed to be absolutely continuous with respect to the Lebesgue measure and non-increasing  in $\theta$. See Section \ref{section:2} for precise definitions and notation. The confidence distribution  $C(\theta)$ is a canonical pivot in the sense of being uniformly distributed on the unit interval
when $\hat\theta$ is distributed according to $\theta$.  
When $\mle$ is a sufficient statistic with monotone likelihood ratio, $C(\theta)$ is also optimal in the Neyman-Pearson sense, that is it describes smaller confidence intervals at a given level when compared to any other confidence distribution for the parameter $\theta$ \citep[Section 5.4]{Sch16}.
An \equitailed confidence curve is readily obtained from $C(\theta)$ by $cc(\theta)=|1-2C(\theta)|$. 

In this paper we shall be concerned with confidence curves obtained from the log-likelihood ratio $w(\theta)$, 
and we shall study the properties of median bias correction. Median bias correction of a confidence curve, proposed by \citet{Sch07}, is a  method to make the resulting confidence curve approximately tail--symmetric. 
In the normal example $w(\theta)=((\hat \theta-\theta)/\sigma)^2$ and the confidence curve mentioned above is also given by $cc(\theta)=Q(w(\theta))$ where $Q$ is the cumulative chi-square distribution function with one degree of freedom. This confidence curve is tail-symmetric, as mentioned, and the confidence interval of level $0$ is the single point $\mle$ which thus has median $\theta$ and is said to be median unbiased. In general $w(\theta)$ hits zero at the maximum likelihood estimator, which might not be median unbiased. Let $\mle$ have median $b(\theta)$. The median bias corrected confidence curve is the confidence curve of the parameter $b(\theta)$. The idea is to probability transform the bias corrected log likelihood ratio $w^* (\theta)=w(b(\theta))$ rather than $w(\theta)$. With $F^*(y;\theta)$ denoting the sampling distribution of $w^*(\theta)$ when the data is distributed according to $\theta$, the bias corrected confidence curve is $cc^*(\theta)=F^*(w^*(\theta);\theta)$. Since $\argmin(cc^*(\theta))=b^{-1}(\mle)$ is median unbiased, the level set at $\alpha=0$  is typically the single point   $b^{-1}(\mle)$ and, by continuity,  $cc^*(\theta)$ is close to be \equitailed at low levels. We undertake a theoretical study of the asymptotic properties of $cc^*(\theta)$ by showing that $cc^*(\theta)$ is third-order tail-symmetric for $(\theta,\mle)$ in the normal deviation range in two important classes of parametric models with parameter dimension one. First, we consider parametric models that belong to the Efron's normal transformation family \citep{Efr82}. Then, we extend the result to regular one dimensional exponential families, where we also discuss the relation between median bias corrected and modified directed likelihood of \citet{Bar86}, thus providing an alternative approximation to the latter. 
Since median bias correction works so well in these cases, it is reasonable to expect the method to work well quite generally. However, when a canonical confidence distribution is available, as in the exponential family models, we do of course not advocate to use median bias correction rather than using the canonical confidence distribution.

The rest of the paper is organized as follows. In Section \ref{section:2}, we recast confidence estimation in terms of confidence curves and introduce the notation we use in the sequel. We also define the confidence curve based on inverting the median bias corrected version of the log-likelihood ratio. In Section \ref{section:3} and \ref{section:4}, we investigate its asymptotic properties in terms of tail symmetry in the Efron's normal transformation family and in one dimensional exponential families, respectively. Finally, in Section \ref{section:5} some concluding remarks and lines of future research are presented, 
together with an example that provides a preliminary illustration of the use of median bias correction in the presence of nuisance parameters.
Some proofs and a technical lemma are deferred to the Appendix.

\section{Likelihood-based confidence curves}\label{section:2}

Let $X=(X_1,\ldots,X_n)$ be a continuous random sample with density $f(x;\theta)$ 
depending on a real parameter $\theta\in\Theta\subset\RR$ and let $\Prob(\,\cdot\,;\theta)$ indicate probabilities calculated under $f(x;\theta)$. The log-likelihood is $\ell(\theta)=\ell(\theta;x)=\log f(x;\theta)$, and the log-likelihood ratio is $w(\theta)=w(\theta;x)=2(\ell(\hat \theta;x)-\ell(\theta;x))$, where $\mle$ is the maximum likelihood estimate. We drop the second argument in sample-dependent functions like $w$ and $\ell$ whenever it is clear from the context whether we refer to a random quantity or to its observed value. Unless otherwise specified, all asymptotic approximations are for $n\to\infty$ and stochastic term $O_p(\cdot)$ 
refers to convergence in probability with respect to $f(x;\theta)$.  We  assume that the model is sufficiently regular for the validity of first order asymptotic theory, cfr. \citet[Chapter 3]{Bar94}. In particular, $w(\theta)$ converges in distribution to a chi-squared random variable, hence, by contouring $w(\theta)$ with respect to this distribution we obtain intervals of $\theta$ values given by the level sets for the curve $Q(w(\theta))$ where $Q$ is the distribution function of the chi-squared distribution with 1 degree of freedom. This curve depends on the sample $x$ and has its minimum at $\mle$. However its level sets are not in general exact confidence intervals since the chi-squared approximation for the distribution of $w(\theta)$ is valid only for $n$ large and the coverage probabilities equal the nominal levels only in the limit. As a consequence, $Q(w(\theta))$ is not uniformly distributed on the unit interval under $P(\,\cdot\,;\theta)$, a property we require for a {\it regular} confidence curve as spelled in the following definition.

%
%
\begin{definition}\label{definition:1}
A function $cc:\Theta\times\RR^n\to[0,1)$ is a regular confidence curve when $\min_\theta\,cc(\theta;x)=0$, the level sets $\{\theta: cc(\theta;x) \leq 1-\alpha\}$ are finite intervals for all $\alpha\in(0,1)$, and $cc(\theta;X)\sim\mbox{Unif}(0,1)$ under $P(\,\cdot\,;\theta)$. 
\end{definition}

Confidence curves might be defined for parameters of higher dimension and also for irregular curves that even might have more than one local minimum or might have infinite level sets for $\alpha<1$, see \citet[Section 4.6]{Sch16}. Note that, under Definition \ref{definition:1}, $I=\{\theta:\,cc(\theta;x)\leq 1-\alpha \}$ is an exact confidence region of level $1-\alpha$ since $\Prob(I\ni\theta;\theta)=\Prob(cc(\theta;X) \leq 1-\alpha;\theta) =1-\alpha$. 
Among confidence curves, of special importance are confidence distributions, which are confidence curves that describe upper confidence limits. The definition is as follows.
%
%
\begin{definition}\label{definition:1b}
A function $C:\Theta\times\RR^n\to[0,1)$ is a confidence distribution when $C(\cdot; x)$ is a cumulative distribution function in $\theta$ for all $x$ and $C(\theta;X)\sim\mbox{Unif}(0,1)$ under $P(\,\cdot\,;\theta)$. 
\end{definition}
Keep in mind that the realized confidence curve and confidence distribution depend on the data, and prior to observation they are random variables (with distribution depending on the parameter value from which the data are generated). To keep the notation simple, we drop the second argument $x$ in $cc(\theta;x)$ and $C(\theta;x)$. Moreover, we will confine ourselves to regular confidence curves 
$cc(\theta)$ with only one local minimum. In this setting $cc(\theta)$ can be transformed into a  distribution via
\begin{equation}\label{eq:H}
  H(\theta)=\Small{\frac{1}{2}}\{1-\mbox{sign}
  (\tilde\theta-\theta)cc(\theta)\},\quad
  \tilde\theta=\arg\min\nolimits_\theta cc(\theta)
\end{equation}  
so that the left and right endpoints of the interval $I=\{\theta:\ cc(\theta)\leq 1-\alpha\}=(\underline\theta,\,\bar\theta)$, are given by
\begin{equation}\label{eq:endpoint}
  \underline\theta
  =H^{-1}\left(\Small{\frac{\alpha}{2}}\right),
  \quad  
  \bar\theta
  =H^{-1}\left(1-\Small{\frac{\alpha}{2}}\right),
\end{equation}  
respectively. 
We refer to $\tilde\theta=H^{-1}(\frac{1}{2})$ as 
the median confidence estimator for $\theta$.
By construction, $\underline\theta$ and $\bar\theta$ satisfy
 $\Prob(\underline\theta>\theta;\theta)+
 \Prob(\bar\theta<\theta;\theta)=\alpha$. We then say that $cc(\theta)$ is tail--symmetric when the interval 
  $(\underline\theta,\bar\theta)$ 
is \equitailed, that is  
  $$\Prob(\underline\theta>\theta;\theta)=
  \Prob(\bar\theta<\theta;\theta)=\alpha/2,\quad
  \forall \alpha\in(0,1).$$ 
This is equivalent to $H(\theta)$ defining a confidence distribution according to Definition \ref{definition:1b}.
%
%
\begin{definition}\label{definition:2}
A confidence curve $cc(\theta)$ is tail--symmetric if
  $H(\theta)$
in \eqref{eq:H} is a confidence distribution according to Definition \ref{definition:1b}.
\end{definition}
The relation obviously works in the other direction: given a confidence distribution $C(\theta)$, 
  $cc(\theta)=1-2\min\{C(\theta),1-C(\theta)\}=|1-2C(\theta)|$
defines a tail-symmetric confidence curve, see \citet{Bir61}. Note that the median confidence estimator $\tilde\theta$ of a tail--symmetric $cc(\theta)$ is median-unbiased, i.e.~$\Prob(\tilde\theta>\theta;\theta)=0.5$. See \citet[Section 5.6]{Leh86}. The relation between median-unbiased estimators and \equitailed intervals have been noted by \citet{Sko89} in connection with the maximum likelihood estimator. 

We now focus on confidence distributions derived from the likelihood. It is convenient to set as the exact confidence distribution the one obtained from the sampling distribution of the maximum likelihood estimator, namely
\begin{equation} \label{eq:CD}
  C(\theta)=1-G(\mle;\theta)\,,
\end{equation}
where we assume that $G$, the distribution function of $\mle$, is continuous and non-increasing in $\theta$. In order to have $C(\theta)$ being a proper cumulative distribution function, it is also required that 
  $\lim_{\theta\downarrow a}G(\mle;\theta)=0$
and
  $\lim_{\theta\uparrow b}G(\mle;\theta)=1$,
where $a$ and $b$ are the infimum and supremum of the parameter space $\Theta$, respectively. 
The $\alpha$-quantile
is denoted by $\mle(\alpha)= C^{-1}(\alpha)$. In particular, $\mle(\half)$ corresponds to the median--unbiased estimator of $\theta$. The exact distribution $G(y;\theta)$ is generally unknown and the asymptotic approximation of the confidence limit $\mle(\alpha)$ has been object of an extensive research which goes beyond first order accuracy. See \citet{DiC96} for a review. Third--order approximations to $G(y;\theta)$, and thus to  $C(\theta)$, can be obtained from the modified directed likelihood of \citet{Bar86}, see Section 4.2 for a discussion. We will instead look for a route to such good approximations by transforming the scale at which the log-likelihood ratio is presented. To this aim, let $F(y;\theta)$  be the  sampling distribution function of $w(\theta)$ under $P(\,\cdot\,;\theta)$, and define
\begin{equation}\label{eq:conf-curve}
  cc(\theta)= F\big( w(\theta);\theta\big)\,.
\end{equation} 
According to \eqref{eq:H}--\eqref{eq:endpoint}, $\underline\theta$ and $\bar\theta$ are the endpoints of a confidence interval of level $1-\alpha$. It is clear that, in general, $cc(\theta)$ is not tail--symmetric according to Definition \ref{definition:2}, in particular when $\mle$ is not median unbiased. 
More generally, the distribution estimator 
\begin{equation} \label{eq:CD_w}
  H(\theta)=\half\{1-\mbox{sign}(\mle-\theta)
  F\left( w(\theta);\theta\right)\}
\end{equation}  
is not uniformly distributed on the unit interval under $P(\,\cdot\,;\theta)$. According to first-order asymptotics, $cc(\theta)$ is tail--symmetric up to the first order of approximation, that is
\begin{equation}\label{eq:first-asymp2}
  H(\theta)=C(\theta)+O_p(n^{-1/2})\,.
\end{equation} 
Consequently
  $\underline\theta=
  \mle(\Small{\frac{\alpha}{2}})+O_p(n^{-1})$
and
  $\bar\theta=
  \mle(1-\Small{\frac{\alpha}{2}})+O_p(n^{-1}),$
and $\Prob(\underline\theta>\theta;\theta)=\Prob(\bar\theta<\theta;\theta)+O(n^{-1/2})$. 

In order to improve on (\ref{eq:first-asymp2}), we consider the median bias  correction to $w(\theta)$
. Let $b(\theta)$ be the median of $\mle$ as function of $\theta$, that is 
  $$G(b(\theta);\theta)=0.5.$$ 
By assumption, $b(\theta)$ is continuously  increasing in $\theta$ and  
  $b^{-1}(\mle)=\mle(\half)$,
as a simple calculation reveals. The {\it median bias corrected log-likelihood ratio} is defined as
\begin{equation}\label{eq:corr-dev}
   w^*(\theta)=w^*(\theta;x)=w\big(b(\theta);x\big)\,,
\end{equation}
By construction, $w^*(\theta)$ attains its minimum at $\mle(\half)$, the median unbiased estimator of $\theta$. Since both the likelihood function $\ell(\theta)$ and the median function $b(\theta)$ are invariant to  monotone parameter transformations, invariance is preserved for $w^*(\theta)$. See \citet{Fir93} for a different type of likelihood correction, aimed at reducing the bias of the maximum likelihood estimator. 
The median bias corrected confidence curve is defined as
\begin{equation}\label{eq:corr-net}
  cc^*(\theta)=F^*\big( w^*(\theta);\theta\big),
\end{equation}
where 
$F^*(y;\theta)$ stands for the sampling distribution of $w^*(\theta)$ under $P(\,\cdot\,;\theta)$. According to \eqref{eq:H}, it yields the distribution estimator
\begin{equation}\label{eq:CD_w*}
  H^*(\theta)=
  \half \left\{1-\mbox{sign}\left(\mle(\half)-\theta\right)
  F^*\big( w^*(\theta);\theta\big)\right\}.
\end{equation}
%
%
For illustration, we consider confidence distributions for the variance parameter $\theta$ in the normal model $N(0,\theta)$.
For $\mle=n^{-1}\sum_{j=1}^n x_i^2$, the log-likelihood ratio is
  $w(\theta)=n(\mle/\theta
  -\log(\mle/\theta)-1)$.
Based on $n\mle/\theta\sim\chi^2_n$, one finds 
  $G(y;\theta)=\Prob(\chi^2_n\leq ny/\theta)$ 
and  
  $b(\theta)=\chi^2_{n,.5}\theta/ n$ (in obvious notation).
Using the pivotal distribution of $n\mle/\theta$, $F(y;\theta)$ and $F^*(y;\theta)$ can be computed via Monte Carlo.
Based on a simulated sample of size $n=10$ with $\theta=4$, 
the left panel of Figure \ref{fig3} displays $cc^*(\theta)$ according to \eqref{eq:corr-net} while the right panel reports $C(\theta)$, $H(\theta)$ and $H^*(\theta)$, according to \eqref{eq:CD}, \eqref{eq:CD_w} and \eqref{eq:CD_w*}, respectively. 
\begin{figure}[!ht]
\begin{center}
 \includegraphics[width=\textwidth]{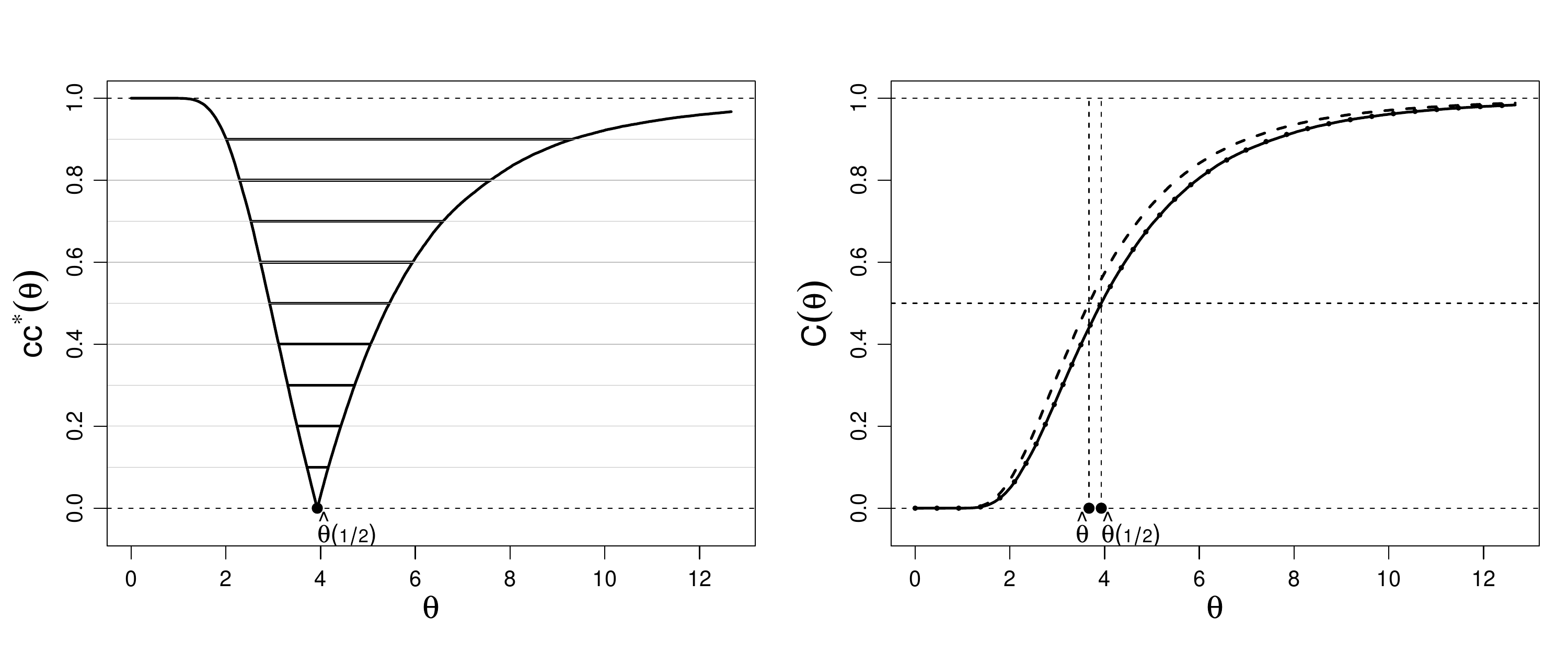}
\caption{\it  
Confidence distributions for $\theta$ in the normal model $N(0,\theta)$ for $n=10$ observations generated according to $\theta=4$. 
Left panel: $cc^*(\theta)$ together with some of its confidence intervals. Right panel: $H(\theta)$ (dashed line), $H^*(\theta)$ (solid line), and nearly on top $C(\theta)$ (dotted line). $cc^*(\theta)$, $H(\theta)$ and $H^*(\theta)$ are based on 50000 Monte Carlo simulations.} \label{fig3}
\end{center}
\end{figure}
%
%
Note that $H^*(\theta)$ and $C(\theta)$ are on top of each other and are almost indistinguishable. Hence, the median correction in \eqref{eq:corr-net}, by making the median confidence estimator of $cc^*(\theta)$ coincide with $\mle(\half)$, shifts the whole curve $H^*(\theta)$ towards $C(\theta)$, thus inducing nearly exact tail symmetry. 
We return to this example in Section \ref{section:4} where we give a theoretical justification to the fact $H^*(\theta)$ and $C(\theta)$ coincide to the third order of approximation in Theorem \ref{theorem:2}.

We conclude this section by noting that, while $w^*$ can be interpreted as the log-likelihood ratio for the parameter $\psi=b^{-1}(\theta)$, that is $w^{(\psi)}(\psi)=w(b(\psi))=w^*(\psi)$, $cc^*(\psi)$ does not correspond to the confidence curve in the $\psi$-parametrization, that is 
  $cc^{(\psi)}(\psi):=F(w^{(\psi)}(\psi;x);\psi)$
where $F(y;\psi)$ stands now for the sampling distribution of the log-likelihood ratio in terms of the $\psi$ parameter. 
As an example, consider the exponential model, $f(x;\theta)=\theta\edr^{-\theta x}$. By standard calculation one finds that $w(\theta)=2n(\theta/\mle-1-\log(\theta/\mle))$, $2n\theta/\mle\sim\chi^2_{2n}$ and $b(\theta)=2n\theta/\chi^2_{2n,.5}$. Hence, for $\psi=\chi^2_{2n,.5}\theta/2n$, we get 
  $cc^*(\psi)=\Prob(2n\,h(X/\chi^2_{2n,.5})\leq w(b(\psi))$
while
  $cc^{(\psi)}(\psi)=\Prob(2n\,h(X/2n)\leq w(b(\psi))$,
where $X\sim \chi^2_{2n}$ and $h(x)=x-1-\log(x)$.
On the other hand, $cc^*(\theta)$ shares with $cc(\theta)$ the property of invariance with respect to monotone transformation of the parameter: if $\psi=g(\theta)$ for $g$ invertible, then, it is easy to see that bias corrected confidence curve in the $\psi$--parametrization, say $cc^{*\,(\psi)}(\psi)$, corresponds to $cc^*(g^{-1}(\psi))$. 
This can be easily verified in the exponential model above by taking, e.g., $g(\theta)=1/\theta$ so that $\psi$ represents the mean parameter. 
In the sequel, for ease of notation, we avoid superscripts as in $w^{(\psi)}$ and $cc^{(\psi)}$ whenever the parametrization the likelihood is referring to will be clear from the context.


\section{Normal transformation family}\label{section:3}
In this section we establish third--order tail symmetry of the bias corrected confidence curve $cc^*(\theta)$ when $\mle$ is a sufficient statistic and belongs to the normal transformation family of \citet{Efr82}. This family of distributions was used by \citet{Efr87} to introduce bias and acceleration corrected bootstrapped confidence intervals that achieve second order accuracy. The idea is that standard intervals are based on assuming that the normal approximation of $(\mle-\theta)/\hat\sigma$ is exact, with $\hat\sigma$ a fixed constant and, hence, convergence to normality can be improved by considering a monotone transformation of $\mle$ and $\theta$ which is exactly normalizing and variance stabilizing. Second order accuracy was later extended to regular statistical models such as the exponential family, see \citet{DiC92}. 
We follow a similar path here, as we first prove, in Theorem \ref{theorem:1}, tail symmetry in the normal transformation family as this case provides a simple illustration of the generalized inverse mapping argument reported in Lemma \ref{lemma:1} of the Appendix. Theorem \ref{theorem:2} of Section \ref{section:4} addresses tail symmetry in the exponential family, where an additional Cornish--Fisher expansion of the distribution of the maximum likelihood estimator is needed. Theorem \ref{theorem:2} is indeed a more general result than Theorem \ref{theorem:1} since, by Pitman-Koopman-Darmois Theorem, cfr. \citet[Theorem 6.18]{Leh:Cas:98}, if the data are independent and identically distributed and the dimension of the sufficient statistic does not depend on $n$, as we are assuming here, then the model is an exponential family.

Let $\mle$ be a sufficient estimator for $\theta$, not necessarily maximizing the likelihood, but behaving asymptotically like the maximum likelihood estimator in terms of order of magnitude of its bias, standard deviation, skewness, and kurtosis: 
\begin{equation}\label{eq:5.1}
  \mle-\theta\sim(B_\theta/n,C_\theta/n^{1/2},
  D_\theta/n^{1/2}, E_\theta/n)\,,
\end{equation}
where $B_\theta,C_\theta,D_\theta$, and $E_\theta$ are functions of $\theta$ and $n$ (the latter suppressed in the notation) bounded in $n$. See equations (5.1)--(5.3) in \citet{Efr87}. 
Next, suppose there exists a monotone increasing transformation $g$ and constants $z_{_0}$ (bias constant) and $a$ (acceleration constant) such that
 $\hat\phi=g(\mle)$
 and
 $\phi=g(\theta)$
satisfy
\begin{equation}\label{eq:norm-transf}
 \hat\phi=\phi+(1+a\phi)(Z-z_{_0}),\qquad Z\sim N(0,1)\,,
\end{equation}
where $\phi>-1/a$ when $a>0$ and $\phi<-1/a$ when $a<0$. 
Model \eqref{eq:norm-transf} has standard deviation linear in $\phi $ on the transformed scale. It provides a pivot with accompanying confidence distribution $C^{(\phi)}(\phi)$. The latter is directly transformed back to a confidence distribution for $\theta$, that is $C^{(\theta)}(\theta)=C^{(\phi)}(g(\theta))$. 
Theorem \ref{theorem:1} states that $cc^{*\,(\phi)}(\phi)$ as well as $cc^{*\,(\theta)}(\theta)$, are third order tail-symmetric according to Definition \ref{definition:2}, an improvement up to $O_p(n^{-3/2})$ in the asymptotic order displayed in \eqref{eq:first-asymp2}. The proof relies on the asymptotic inversion of convex functions reported in Lemma \ref{lemma:1} in the Appendix.
%
%
\begin{theorem}\label{theorem:1}
Let  $\mle$ be a sufficient estimator of $\theta$ based on a sample of size $n$ satisfying \eqref{eq:5.1}, and assume there exists a monotone increasing function $g$ such that \eqref{eq:norm-transf} holds. Then, for $C(\theta)$ and $H^*(\theta)$ defined in \eqref{eq:CD} and \eqref{eq:CD_w*}, respectively,
\begin{equation}\label{eq:second-asymp2}
  H^*(\theta)=C(\theta)
  +O_{p}\big( n^{-3/2}\big),\quad 
  \mbox{for }
  n^{1/2}(\theta-\mle)/C_\theta=O_p(1)\,.
\end{equation}
\end{theorem}
%
%
\begin{proof}
Since a confidence curve for $\phi =g( \theta) $ translates into one for $\theta $ for the invertible transformation $g$, it is sufficient to prove \eqref{eq:second-asymp2} in the transformed normal model. Under \eqref{eq:5.1}, the normalizing transformation $g$ is locally linear in its argument with a scale factor of order $n^{1/2}$. In particular, from \eqref{eq:norm-transf}, the normal deviation range $n^{1/2}(\theta-\mle)/C_\theta=O_p(1)$ in \eqref{eq:second-asymp2} corresponds to
\begin{equation}\label{eq:range}
  (\phi-\hat\phi)/(1+a\phi)=O_p(1)\,.
\end{equation}  
According to \citet[Theorem 2]{Efr87}, $z_{0}=a[1+O\left(n^{-1}\right)] $ and both $z_{0}$ and $a$ are $O\big(n^{-1/2}\big)$ as long as the $\mle$ satisfies \eqref{eq:5.1}. We will make repeatedly use of these asymptotic behaviors throughout the proof, even though we suppress the dependence of $a$ and $z_0$ on $n$ in the notation.
The log-likelihood 
  $\ell(\phi;\hat\phi)=
  -\big[(\hat\phi-\phi)/(1+a\phi)
  +z_0\big]^2/2-\log(1+a\phi)$ 
is not maximized at $\phi=\hat\phi$, unless $z_0=a$, rather at 
\begin{equation}\label{eq:c}
  \hat\phi^c=\hat\phi-c(1+a\hat\phi),\quad
  c=\frac{1}{a}-\frac{(1-az_0)}{2a^3}\bigg[
  \bigg(1+\frac{4a^2}{(1- az_0)^2}\bigg)^{1/2}-1\bigg]
\end{equation}
as a simple calculation reveals. One finds that $c=a-z_{_0}+O(n^{-3/2})$ and, consequently, $c=O(n^{-3/2})$.
Actually, $\hat\phi^c$ belongs to the normal transformation family \eqref{eq:norm-transf} since it can be written as
  $\hat\phi^c=\phi+(1-ac)(1+a\phi)(Z-z^c_0)$
for
  $z^c_0=z_0+c/(1-ac)$,
see \citet[Section 11]{Efr87}, with distribution
\begin{equation}\label{eq:dfMLE}
  G(y;\phi)=\Phi\bigg({y-\phi\over(1-ac)(1+a\phi)}+z_0^c\bigg)\,,
\end{equation}  
and median function 
  $b(\phi)=\phi-z^c_0(1-ac)(1+a\phi)$.
Note that $b(\phi)$ is increasing in $\phi$ when $z_0^ca(1-ac)<1$, which we assume without loss of generality since it certainly is for large $n$. Since $\hat\phi^c$ is a sufficient statistic, the log-likelihood ratio for $\phi$ is
 $$w(\phi;\hat\phi^c)=-(z^c_0)^2+
 \bigg(\frac{\hat\phi^c-\phi}{(1-ac)(1+a\phi)}+z^c_0\bigg)^2
 -2\log{1+a\hat\phi^c\over1+a\phi}\,.$$
It is easy to check that $w(\phi;\hat\phi^c)$ is convex in both arguments, and so is its bias corrected version $w^*(\phi;\hat\phi^c)=w(b(\phi);\hat\phi^c)$.
Let $H^*(\phi)$ be defined according to \eqref{eq:CD_w*}. We are interested in expressing $H^*(\phi)$ in terms of tail probabilities associated to $\hat\phi^c$ for comparison with the confidence distribution 
 $C(\phi)
 =1-G( \hat\phi^c;\phi )$.
To this aim, let $\hat\phi^*$ be implicitly defined in function of $\hat\phi^c$ and $\phi$ by $w^{\ast }(\phi ;\hat\phi^*) =w^{\ast }\big( \phi ; \hat\phi^c\big)$. Then, 
  $w^*(\phi;x) \leq w^*(\phi;\hat\phi^c)$ 
for  
  $\hat\phi^*\leq x\leq \hat\phi^c$
when $\hat\phi^c>b(\phi)$, 
for  
  $\hat\phi^c\leq x\leq \hat\phi^*$
when $\hat\phi^c<b(\phi )$. 
We only consider the first case, where the equality of interest is
 $H^*(\phi)=\half\{1-cc^*(\phi)\}
 =\half\{1-G(\hat\phi^c;\phi)+G(\hat\phi^*;\phi)\}.$
Hence, for $\hat\phi^c>b(\phi)$, the normal deviation range $n^{1/2}(\theta-\mle)/C_\theta=O_p(1)$ in \eqref{eq:second-asymp2} corresponds to
\begin{equation}\label{eq:pier}
  G(\hat\phi^*;\phi) =1-G(\hat\phi^c;\phi) 
  +O_p(n^{-3/2})  
\end{equation}  
for $(\phi,\hat\phi)$ in \eqref{eq:range}. As for the right hand side of \eqref{eq:pier}, from \eqref{eq:dfMLE} it follows that, when $\hat\phi^c>b(\phi)$,
  $1-G(\hat\phi^c;\phi)=G(2b(\phi)-\hat\phi^c;\phi)$,
so that \eqref{eq:pier} is implied by
\begin{equation}\label{eq:second-asymp5}
 {\hat\phi^*-\phi\over
 (1-ac)(1+a\phi)}= 
  {(2b(\phi)-\hat\phi^c)-\phi\over
 (1-ac)(1+a\phi)}
 +O_p(n^{-3/2})\,,
\end{equation}
for $(\phi,\hat\phi)$ in \eqref{eq:range}. %
In order to establish \eqref{eq:second-asymp5}, we derive an asymptotic expansion of $\hat\phi^*:=\hat\phi^*(\hat\phi^c,\phi)$ locally at $\hat\phi^c=b(\phi)$ by an application of the generalized inverse mapping argument of Lemma \ref{lemma:1}. 
Let
  $f_n(x)=w^*\big(\phi;\,b(\phi)+[1+ab(\phi)]x\big)$
so that $w^*\big(\phi;\hat\phi^c\big)=f_n(x_0)$  for $x_0=[\hat\phi^c-b(\phi)]/[1+ab(\phi)]$. Also, let $g_n(x)$ be implicitly defined by $f_n(x)=f_n(g_n(x))$ so that 
 $g_n(x_0)=[\hat\phi^*-b(\phi)]/[1+ab(\phi)]$.
One finds
  $f_n^{(2)}(0)=2[(1-ac)^{-2}+a^2]$
  and 
  $f_n^{(k)}(0)=(-1)^{k} 2(k-1)!a^k$
  for 
  $k\geq 3$,
so that the hypotheses of Lemma \ref{lemma:1} are satisfied. Hence, $g_n(x_0)=-x_0+O(n^{-3/2})$ for $x_0=O(1)$,
that is
 $${\hat\phi^*-b(\phi)\over 1+ab(\phi)}
 =-{\hat\phi^c-b(\phi)\over
 1+ab(\phi)}+O_p(n^{-3/2})\,,$$
for 
  $[\hat\phi^c-b(\phi)]/[1+ab(\phi)]=O_p(1)$. 
Since 
  $1+ab(\phi)=[1-az_0^c(1-ac)](1+a\phi)
  =(1-az_0-c)(1+a\phi)$ 
and both $1-az_0-c$ and $1-ac$ are $O(1)$, we get  \eqref{eq:second-asymp5} for $[\hat\phi^c-b(\phi)]/(1+a\phi)=O_p(1)$. The latter corresponds to $(\phi,\hat\phi)$ in the normal deviation range \eqref{eq:range} upon substitution for $\hat\phi^c$ and $b(\phi)$.
The proof is then complete.
\end{proof}
%
%
As an illustration of \eqref{eq:second-asymp2}, assume the coefficients in \eqref{eq:norm-transf} are in agreement so that $\hat\phi$ is the maximum likelihood estimator, that is $z_0=a$, see equation \eqref{eq:c}.
Therefore, let
  $\hat\phi=\phi+(1+a\phi)(Z-a)$,
which yields 
 $b(\phi )=\phi-a(1+a\phi)$.
The distribution function of $\hat\phi$ is $G(y;\phi)=\Phi[(y-\phi)/(1+a\phi)+a]$, 
so that
  $C(\phi)
  =1-\Phi[(\hat\phi-\phi)/(1+a\phi)+a]$,
with median--unbiased estimator $\hat\phi(\half)=(\hat\phi+a)/(1-a^2)$. The log-likelihood ratio and its bias corrected version $w^*(\phi;\hat\phi)$ are easily derived. The exact distribution $F^*$ of $w^*$ is recovered from the standard normal distribution,
  $w^*(\phi;\hat\phi)
  \sim -a^2+[Z/(1-a^2)+a]^2
  -2\log[1+aZ/(1-a^2)]$,
and $cc^*(\phi)$ can be calculated by Monte-Carlo. In the left panel of Figure \ref{Normal_transformation} we plot $C(\phi)$ and $H^*(\phi)$ for $\hat\phi=10$ and $a=0.3$. Even for a non-negligible acceleration $a$ (later we argue that $a=O(n^{-1/2})$, so it roughly corresponds to $n=10$), the median corrected confidence curve $cc^*(\phi)$ nearly exactly recovers, through $H^*(\phi)$, the confidence distribution $C(\phi)$. The right panel shows that the difference between the two confidence distributions is very small, approximately of order $n^{-3}$, suggesting that the order of magnitude in \eqref{eq:second-asymp2} might be conservative.
\begin{figure}[!ht]
\begin{center}
 \includegraphics[width=\textwidth]{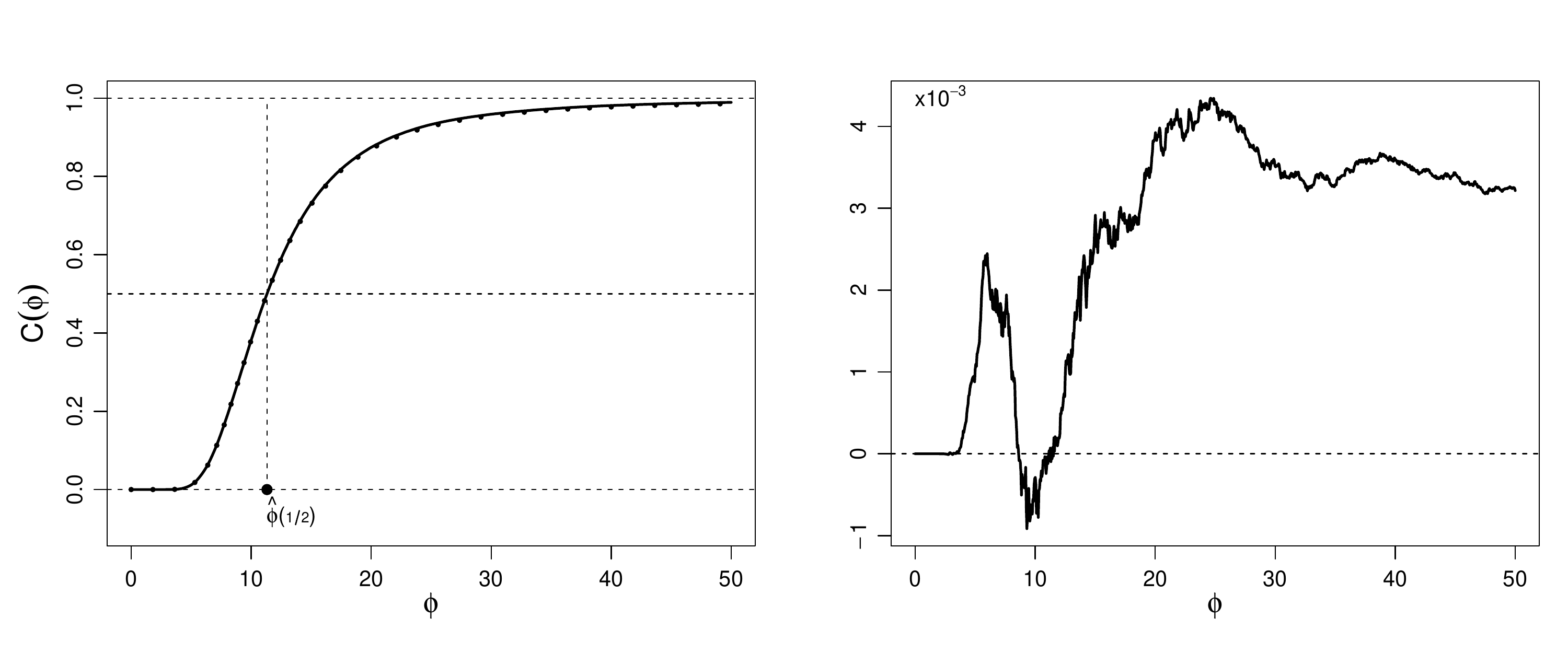}
\caption{\it Normal transformation model \eqref{eq:norm-transf} with $a=z_{0}=.3$, $\hat\phi=10$. Left panel: confidence distributions $H^*(\phi)$ (solid line), and nearly on top $C(\phi)$ (dotted line). Right panel: difference $H^*(\phi)-C(\phi)$. $H^*$ is based on 100000 Monte Carlo simulations. }
\label{Normal_transformation}
\end{center}
\end{figure}


\section{Exponential family}\label{section:4}
\textit{4.1 Tail symmetry.}$\quad$ 
In this section we establish third--order tail symmetry for the mean value parameter of regular one-parameter exponential families. Following \citet[Section 5]{DiC92}, let 
  $X\sim f(x;\bar\eta)=\exp[
  \bar\eta t(x)-\bar\psi(\bar\eta)-d(x)]$,
where $\bar\eta$ is an unknown real parameter.
Given a random sample of size $n$, the log-likelihood for $\bar\eta$ based on $y=n^{-1}\sum_{i=1}^n t(x_i)$ has form $\ell(\bar\eta;y)=n[\bar\eta y-\bar\psi(\bar\eta)]$.
Upon defining $\eta=n\bar\eta$ and $\psi(\eta)=n\bar\psi(\bar\eta)$, the log-likelihood for $\eta$ is
  $\ell(\eta;y)=\eta y-\psi(\eta)$.
Since the cumulant generating function for $y$ is
  $\log E(\edr^{\xi y})=\psi(\eta+\xi)-\psi(\eta),$
the $r$-th order cumulant of $y$ is $\psi^{(r)}(\eta)$, the $r$-th order derivative of $\psi(\eta)$. We set $\theta=\psi'(\eta)$, so that $\theta=E(y)$ and $\mle=y$. Consequently, $\sigma_\theta=\psi^{(2)}(\eta)^{1/2}$ is the standard error of $\mle$, where we use the subscript $\theta$ in $\sigma_\theta$ to highlight the dependence on $\theta$. Note that $\sigma_\theta=O(n^{-1/2})$ since $\psi^{(k)}=O(n^{1-k})$. 
The following result can be stated.
%
%
\begin{theorem}\label{theorem:2}
Let $\mle$ and $w(\theta)$ be the maximum likelihood estimator and the log-likelihood ratio for the mean value parameter in a continuous one-dimensional exponential model based on a random sample of size $n$. Also, let $\sigma_\theta$ be the standard error of $\mle$ and $C(\theta)$ and $H^*(\theta)$ be defined in \eqref{eq:CD} and \eqref{eq:CD_w*}, respectively. Then, as $n\to\infty$,
\begin{equation}\label{eq:1}
  H^*(\theta)=C(\theta)
  +O_{p}\big( n^{-3/2}\big),\quad 
  \mbox{for }
  (\theta-\mle)/\sigma_\theta=O_p(1)\,.
\end{equation}
\end{theorem}
The proof is deferred to the Appendix and we only provide here in this paragraph a sketch. Reasoning as in the proof of Theorem \ref{theorem:1}, take $\mle>b(\theta)$ so that
  $H^*(\theta)=
  [1-G(\mle;\theta)+G(\mle^*;\theta)]/2$,
where $\mle^*$ is implicitly defined by $w^*(\theta;\mle^*)=w^*(\theta;\mle)$. When  for $\mle>b(\theta)$, \eqref{eq:1} corresponds to
  $G(\mle^*;\theta)=
  1-G(\mle;\theta)+O_p(n^{-3/2})$,
cfr. \eqref{eq:pier}. An asymptotic expansion of $\mle^*$ as function of $\mle$ and $\theta$ is obtained via the generalized inverse mapping result of Lemma \ref{lemma:1}. In order to work with left tail probabilities, we further define $\mle^{**}$ as function of $\mle$ and $\theta$ to satisfy
  $1-G(\mle;\theta)=
  G(\mle^{**};\theta)$. 
While for the normal transformation family $\mle^{**}$ can be derived in explicit form, in the present setting an additional asymptotic expansion is required. This is achieved by an Edgeworth expansion and the allied Cornish--Fisher inversion for the standardized distribution of $(\mle-\theta)/\sigma_\theta$. The proof is completed by showing that $\mle^*$ and $\mle^{**}$ coincide up to the required asymptotic order for $(\mle-\theta)/\sigma_\theta=O_p(1)$.

Note that tail symmetry of $cc^*(\theta)$ in the $\mbox{N}(0,\theta)$ example of Section \ref{section:2} holds by Theorem \ref{theorem:2} for $\sigma_\theta=\sqrt 2\theta/n^{1/2}$, and, hence, a theoretical justification of what we observed in Figure \ref{fig3} is obtained. It is worth noting that the chi-squared distribution of the maximum likelihood estimator is the running example in \citet{Efr87} where it is shown that the transformation $g$ leading to \eqref{eq:norm-transf} nearly exists, see remark E in Section 11 therein (actually, \citet{Efr87} considers the sampling distribution of the unbiased estimate of the variance when the mean is unknown).
Hence, this example also provides an illustration of tail symmetry in the normal transformation family as stated in Theorem \ref{theorem:1}.
\medskip


\textit{4.2 Comparison with the modified directed likelihood.}$\quad$ 
We adopt here the notation in \citet[Section 5]{Bar94} for the partial derivatives of $\ell$ with respect to $\theta$ and $\mle$, namely
  $$\ell_{k;s}(\theta;\mle)
  =\frac{\partial}{\phantom{\mle}\partial\theta^k}
  \frac{\partial}{\partial\mle^s}
  \ell(\theta;\mle)$$
for nonnegative integers $k$ and $s$. We also adopt the convention of a slash through $\ell$ indicating the substitution of $\theta$ for $\mle$ and a hat sign indicating the substitution of $\mle$ for $\theta$ after any differentiation. The observed information is defined either as $\cancel{j}=-\cancel{\ell}_{2}$ or as $\hat{j}=-\hat{\ell}_{2}$ according to whether it is considered as a quantity depending on the parameter or the data only. Recall the definition $w(\theta;x)=2(\ell(\mle;x)-\ell(\theta;x))$ the log-likelihood ratio and 
\begin{equation}\label{eq:directlike1}  
  r(\theta)=r(\theta;x)=\mbox{sign}(\mle-\theta)w(\theta;x)^{1/2}
\end{equation}  
for the directed likelihood. Since $r(\theta)$ is increasing in $\mle$, 
  $1-P(r(\theta;X)\leq r(\theta;x);\theta)=C(\theta)$,
where $C(\theta)$ has been defined as 
  $C(\theta)=1-G(\hat \theta;\theta)$
assuming that the mle $\hat\theta$ has distribution $G(y,\theta)=\Prob(\hat\theta\leq y;\theta)$ non-increasing in $\theta$.
The modified directed likelihood is defined as
\begin{equation}\label{eq:directlike2}  
  r^*(\theta)=r(\theta)
  -\frac{1}{r(\theta)}
  \log\frac{r(\theta)}{u(\theta)},\quad
  u(\theta)=\widehat{j}\,\{\hat\ell_{;1}
  -\ell_{;1}(\theta)\}^{1/2},
\end{equation}  
see \citet[Section 6.6]{Bar94}. It is a higher order pivot, that is it has normal distribution with error $O(n^{-3/2})$ in the normal deviation range $\sqrt n(\mle-\theta)=O_p(1)$, so that
  $1-\Phi(r^*(\theta))=C(\theta)+O(n^{-3/2})$.
Consider now the median bias corrected directed likelihood,
\begin{equation}\label{eq:directlike3}  
  r(b(\theta))
  =\mbox{sign}(\mle-b(\theta))\,
  w(b(\theta))^{1/2}
\end{equation}
where we recall that $b(\theta)$ has been defined as the median of the $\mle$, i.e. the function of $\theta$ that satisfies $G(b(\theta);\theta)=0.5$. Notice that, since $r(b(\theta))$ is increasing in $\mle$, we also have that 
  $1-P(r(b(\theta);X)\leq r(b(\theta);x);\theta)=C(\theta)$.
In the next theorem we establish that, in regular one parameter exponential families, $r(b(\theta))$ and the modified directed likelihood $r^*(\theta)$ are second order equivalent in the normal deviation range.
%
%
\begin{theorem}\label{theorem:3}
Let $\mle$ and $w(\theta)$ be the maximum likelihood estimator and the log-likelihood ratio for the mean value parameter in a continuous one-dimensional exponential model based on a random sample of size $n$. Also, let $\sigma_\theta$ be the standard error of $\mle$ and $r^*(\theta)$ and $r(b(\theta))$ be defined according to \eqref{eq:directlike1}--\eqref{eq:directlike3}. Then, as $n\to\infty$,
\begin{equation}\label{eq:asy_r}
  r(b(\theta))=r^*(\theta)+O_p(n^{-1}),\quad 
  \mbox{for }
  (\theta-\mle)/\sigma_\theta=O_p(1)\,.
\end{equation}
\end{theorem}
The proof is deferred to the Appendix. Note that, because of the higher order pivotal property of $r^*(\theta)$, \eqref{eq:asy_r} implies that 
  $1-\Phi(r(b(\theta))=C(\theta)+O_p(n^{-1})$
in the normal deviation range, that is $r(b(\theta))$ has sampling distribution closer to normality than $r(\theta)$. 


\section{Discussion}\label{section:5}

There has been a renewed interest in confidence distributions in recent years, see \citet{Xie11} and \citet{Sch16}. In this paper we have undertaken an asymptotic investigation of the merits of median bias correction in deriving higher order accurate confidence limits. We found that, in regular one-dimensional models, the confidence distribution obtained from the bias corrected log-likelihood ratio is third--order equivalent to the unique exact confidence distribution based on the maximum likelihood estimator. Moreover, the bias corrected directed likelihood provides a second order approximation to the modified directed likelihood of \citet{Bar86}, thus consisting in a high order pivot. It shows, from a different perspective, the close connection between the log-likelihood ratio and the distribution of the maximum likelihood estimator so extensively studied in the literature, a key example being the $p^*$ approximation of \citet{Bar83}. We are not aware of similar results in the literature on higher order asymptotics. 

We have not discussed the effect of the bias correction on the sampling distribution of the log-likelihood ratio $w^*(\theta)=w(b(\theta))$. With the median function $b(\theta)$ at hand, the chi-squared transformation of $w^*$ will typically provide more \equitailed intervals than the usual chi-squared calibration of $w$. We found however that median bias correction is second order equivalent to what is found via the modified directed likelihood. The convergence to chi-squared distribution of $w^*$ is thus at least to the second order in regular one-parameter exponential families. 
%

An important direction for future research is the extension of the results of Theorem \ref{theorem:2} to models with nuisance parameters. In full $p$-dimensional exponential models when the interest parameter $\theta$ is a linear function of the canonical parameters, or a ratio of two canonical parameters,
a reparametrization from the canonical parameter vector $\eta$ to $(\theta,\lambda)$, where $\lambda$ is a $(p-1)$-dimensional nuisance parameter, can be made and the canonical statistic $y$ can be re-expressed as $(y_1,y_2)$ having density
  $f(y_1,y_2;\theta,\lambda)=
  \exp[\theta y_1+\lambda y_2-\psi(\theta,\lambda)
  -d(y_1,y_2)]$.
Exact inference on $\theta$ can be based on the conditional distribution of $y_1$ given $y_2$, which depends on $\eta$ only through $\theta$. 
See \citet{Pie92}, and \citet{Sch16} who find the conditional confidence distribution to be uniformly most powerful. The definition of $C(\theta)$ and $b(\theta)$ are to be interpreted conditionally on $y_2$ as well. We expect the median bias corrected confidence curve based on the profile likelihood to be tail-symmetric to the third order, and to the second order to be chi-square distributed. 
The investigation of the relation of the bias corrected profile likelihood with other versions of adjusted profile likelihoods that have been proposed in the literature would also be of interest.
Outside the exponential family, the evaluation of sample space derivatives of the likelihood requires the identification of an ancillary statistic. Moreover, the distribution of the maximum likelihood estimator has to be evaluated conditionally upon this statistic. 
The asymptotic approximations used in Theorem \ref{theorem:2} can be adapted to this setting, a natural extension being for transformation families. 
Next is a preliminary illustration of the use of median bias correction to confidence curves in a multidimensional statistical model. The model in the example below is not in the exponential family, nor an ancillary statistic is available, and we there use brute force to handle the nuisance parameter.
\medskip


\textit{Example.}$\quad$ 
We consider the \comillas{Bolt from heaven} data example from Section 7.4 in \citet{Sch16}. Data consists of $n=195$ winning times in the fastest $100$-m races from 2000 to 2007, that is races that clocked at $10.00$ seconds or better. \citet{Sch16} translate these races results $r_i$ as $x_i=10.005-r_i$ in order to apply extreme value statistics. Specifically, the data is modeled using the Generalized Pareto Distribution (GPD) which has density
  $$f(x;a,\sigma)=\frac{1}{\sigma}\bigg(
  1-a\frac{x}{\sigma}\bigg)^{\frac{1}{a}-1},\quad
  0\leq x\leq \sigma/a,$$
for $a,\sigma>0$. Cfr. Sections 3.4 and 6.5 in \citet{Emb97}. Interest is in estimating 
  $$p=p(a,\sigma)=1-\exp\{-\lambda(1-aw/\sigma)\},$$
for $\lambda=195/8$ and $w=10.005-9.72=0.285$. It takes on the interpretation of the probability, as seen at the start of 2008, that in the $N\sim Pois(\lambda)$ fastest races of 2008 one should experience a race of $9.72$ or better, where $9.72$ is the world record time scored by Usain Bolt on 31 May 2008. See \citet{Sch16} for details. The authors compute a confidence curve for the parameter $p$ by profiling the log-likelihood,
  $\ell_{{\rm P}}(p_0)
  =\max\{\ell(a,\sigma):\ p(a,\sigma)=p_0\}$
and by inverting the profile log-likelihood ratio
  $w(p_0)
  =2(\ell_{{\rm P}}(\hat p)-\ell_{{\rm P}}(p_0))$
with respect to the chi-squared distribution after Bartlett correction,
  $$cc(p)=Q(w(p)/(1+\epsilon)),$$
where $(1+\epsilon)=\E(w(p))\approx 1.07$ (found through simulations) and $Q(\cdot)$ is the chi-squared distribution function with $1$ degree of freedom. By construction, $cc(p)$ points at $\hat{p}=p(\hat{a},\hat\sigma)=0.0345$ according to maximum likelihood estimates $\hat{a}=0.1821$ $(0.0702)$ and $\hat\sigma=0.0745$ $(0.0074)$  (with approximate standard errors in parentheses) and has 90\% confidence interval $[0.0002, 0.1965]$. We proceed next with median bias correction of $w(p)$ so to produce the bias corrected confidence curve $cc^*(p)$. To this aim, the median function of $\hat p$ needs to be estimated. The problem here is that, since $\hat p$ is not a sufficient statistics for $p$, its sampling distribution is not uniquely determined by the value of $p$. If an ancillary statistic for $p$ was available, say $s=s(x)$, then the median function to be used would be
  $$b(p):\ \Prob(\hat p\leq b(p)|s;p)=0.5$$
where conditioning is intended with respect to the realized value of $s$ in the data. Note that $b(p)$ does not depend on the parameter $a$ since the conditional distribution of $\hat p$ given $s$ does not depend on $a$ by definition. The median bias corrected log-likelihood ratio would then be 
  $w^*(p)=w(b(p))$ 
with sampling distribution
  $F^*(y|s;p)=\Prob(w(b(p))\leq y|s;p)$
so that the bias corrected confidence curve would be
  $cc^*(p)=F^*(w^*(p)|s;p)$.

For the sake of illustration we proceed here without relying on such ancillary statistic. It is convenient to introduce the reparameterization $(a,\sigma)\mapsto (a,p)$ so to define the median function of $\hat p$ to depend on both $p$ and $a$, say
  $$b(p|a):\ \Prob(\hat p\leq b(p|a);a,p)=0.5.$$ 
In the left panel of Figure \ref{fig:Bolt} we plot $b(p|a)$ for $a=(0.05,\hat a,0.3)$ (this range has about $95\%$ confidence for $a$). Each curve is obtained by spline interpolation (constrained to be $0$ at $p=0$) of the median of $\hat p$ for a fine grid of $p$ values. For this we used Monte Carlo simulations: for each combination of $a$ and $p$, 15000 samples of size $n=195$ from the GPD were drawn and $b(p|a)$ is estimated via the 15000 realizations of the sample median. Hence we replace $b(p)$ with
  $\hat b(p)=b(p|\hat a)$,
i.e. the solid line plotted in the left panel of Figure \ref{fig:Bolt}. The median bias corrected log-likelihood ratio is then defined as
  $w(\hat b(p))$
and we estimate its sampling distribution 
  $F^*(y;\hat a,p)=\Prob(w(\hat b(p))\leq y;\hat a,p)$
for a grid of $p$ values through simulations. Finally, we compute the bias corrected confidence curve
  $$cc^*(p)=F^*(w(\hat b(p));\hat a,p)$$
and we plot it together with $cc(p)$ (based on the chi-squared approximation with Bartlett correction) in the right panel of Figure \ref{fig:Bolt}. Median bias correction moves the confidence curves slightly to the right
to the effect that the upper 5\% confidence quantile is $0.2278$ instead of $0.1965$.
\begin{figure}[!ht]
\begin{center}
  $$\begin{array}{cc}
  \includegraphics[width=0.45\textwidth]{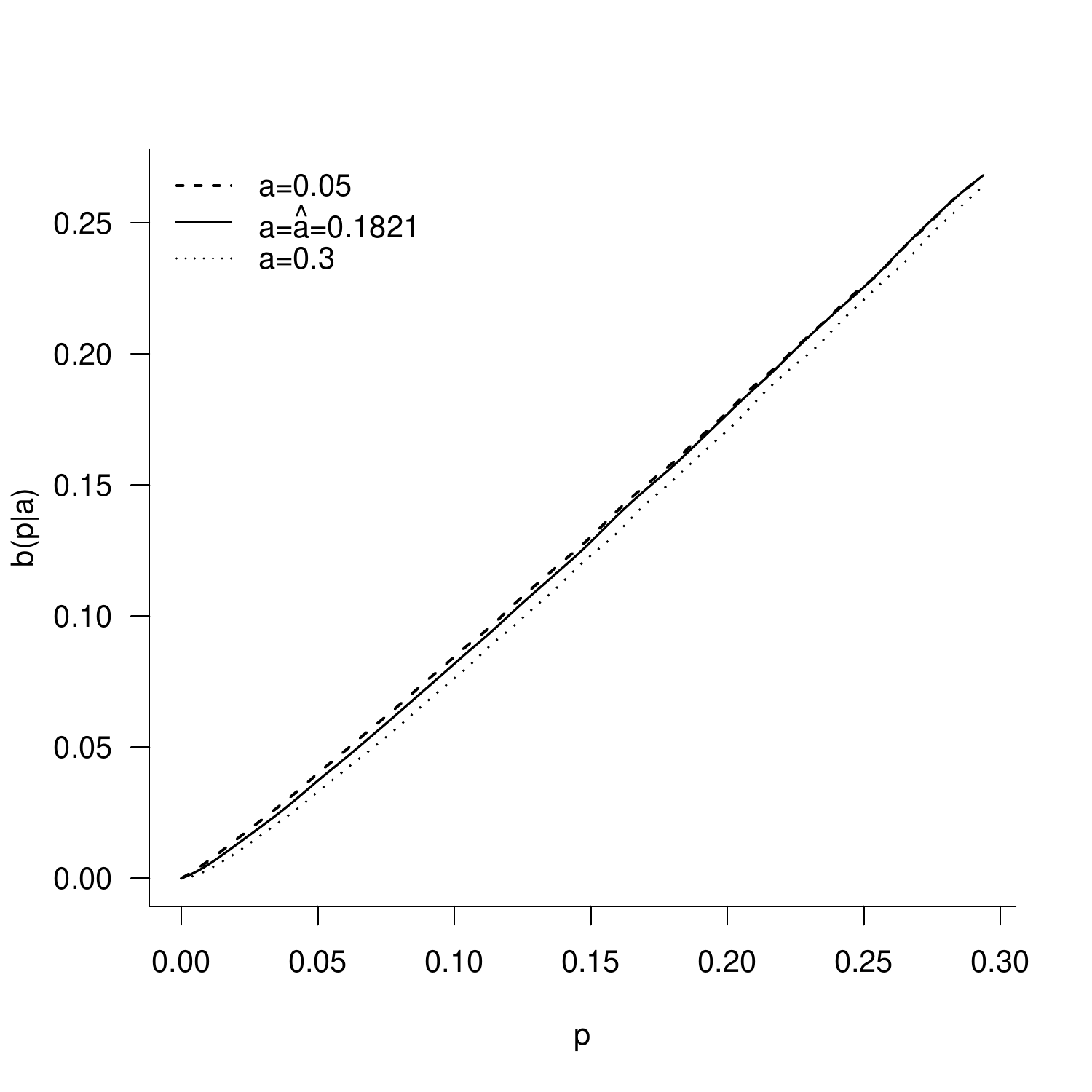}&
  \includegraphics[width=0.45\textwidth]{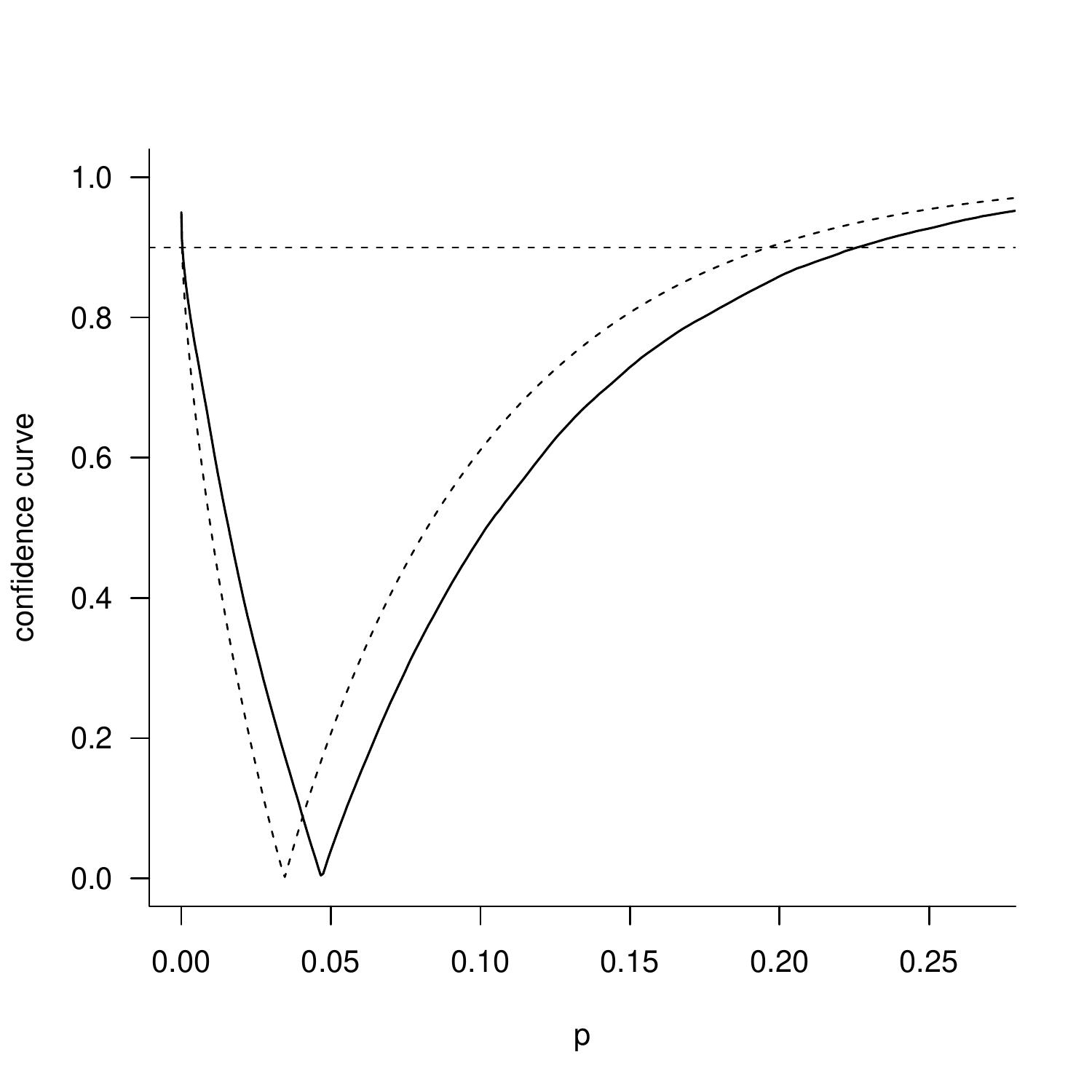}
  \end{array}$$
\caption{\it Bolt from heaven data example with $\hat a=0.1821$, $\hat\sigma=0.0745$ and $\hat p=0.0345$. Left panel: median function $b(p|a)$ for $a=(0.05,\hat a,0.3)$. Right panel: confidence curve $cc(p)$ (dashed line)
together with bias corrected confidence curve $cc^*(p)$ 
(solid line). 
}
\label{fig:Bolt}
\end{center}
\end{figure}

\appendix
\section*{Appendix}
\label{section:appendix}

%
%
\begin{proof of}{Theorem \ref{theorem:2}}
We adopt the notation
  $\eta=\eta(\theta)$
for the inverse of $\psi'(\eta)$ with the convention that, whenever we suppress the dependence of $\eta$ on $\theta$, we mean that it is evaluated at $\theta$. 
Moreover, 
the ``hat'' sign denotes evaluation at $\theta=\mle$ as in $\hat\psi^{(k)}=\psi^{(k)}(\eta(\mle))$ or in $\hat\eta'=\eta'(\mle)$.
Since $\psi^{(k)}(\eta)=O(n^{1-k})$, 
the standardized skewness and kurtosis of $\mle$ are
  $\rho_3=n^{1/2}\psi^{(3)}(\eta)/\sigma^3_\theta$ 
and
  $\rho_4=n\psi^{(4)}(\eta)/\sigma^4_\theta$,
where $\sigma_\theta=\psi^{(2)}(\eta)^{1/2}$ is the standard error of $\mle$. For $U=(\mle-\theta)/\sigma_\theta$ and $K_\theta(x)=\Prob(U\leq x;\theta)$ ,
a Cornish-Fisher expansion of $K_\theta^{-1}(\alpha)$ yields
\begin{equation}\label{eq:corn_fish}
  K_\theta^{-1}(\alpha)=z_\alpha
  +{\rho_3\over 6n^{1/2}}(z_\alpha^2-1)
  +{\rho_4\over 24n}(z_\alpha^3-3z_\alpha)
  +{\rho_3^2\over 36n}(2z_\alpha^3-5z_\alpha)
  +O(n^{-3/2}),
\end{equation}
see \citet[Section 4.4]{Bar89}. 
The following asymptotic expansion of the median of $\mle$ is readily derived:
\begin{equation}\label{eq:median}
  b(\theta)=\theta-\frac{1}{6\sqrt n}\rho_3\sigma_\theta
  +O(n^{-2}).
\end{equation}
The log-likelihood ratio for $\theta$ is 
  $w(\theta;\mle)
  =2[\mle(\hat\eta-\eta)-
  [\hat\psi-\psi(\eta)]]$. 
Let $w^*(\theta;\mle)
  =w(b(\theta);\mle)$ according to \eqref{eq:corr-dev}. 
It is easy to check that the first three sample derivatives of $w^*$ are
  $\partial w^*/\partial\mle
  =2[\hat\eta-\eta(b(\theta))]$, 
  $\partial^2 w^*/\partial\mle^2
  =2	\hat\eta'$ 
and
  $\partial^3 w^*/\partial\mle^3
  =2\hat\eta''$.
Using the formulae for the derivative of the inverse of a function, one obtains
\begin{equation}\label{eq:eta_prime}
  \eta'(x)=1/\psi^{(2)}(\eta(x)),\quad
  \eta''(x)=-\psi^{(3)}(\eta(x))/\psi^{(2)}(\eta(x))^3
\end{equation}
so that $\eta'(\theta)=1/\sigma_\theta^2$ and $\eta''(\theta)=-\rho_3/(n^{1/2}\sigma_\theta^3)$. 
Let $\theta^*$ be implicitly defined as a function of $\mle$ and $\theta$ by $w^*(\theta;\theta^*)=w^*(\theta;\mle)$ and assume that $\mle>b(\theta)$. Then \eqref{eq:1} corresponds to
\begin{equation}\label{eq:4}
  K_\theta\big((\theta^*-\theta)
  /\sigma_\theta \big)
  =1-K_\theta\big(U\big)
  +O(n^{-3/2})
\end{equation}  
for $U=O_p(1)$, cfr. \eqref{eq:pier}. Next, let $\theta^{**}$ be implicitly defined in function of $\mle$ and $\theta$ by
  $1-K_\theta\big(U\big)
  =K_\theta\big((\theta^{**}-\theta)
  /\sigma_\theta \big)$. 
Hence, \eqref{eq:4} is implied by
\begin{equation}\label{eq:5}
  (\theta^*-\theta)/\sigma_\theta
  =(\theta^{**}-\theta)/\sigma_\theta
  +O(n^{-3/2})
\end{equation}  
cfr. \eqref{eq:second-asymp5}. 
We prove \eqref{eq:5} by matching the asymptotic expansions of $\theta^*$ and $\theta^{**}$ via an application of Lemma 1 and Edgeworth expansion of $K_\theta$, respectively.

As for the former, we define, according to Lemma 1, 
  $f_n(x)=w^*(\theta;\,b(\theta)
  +\sigma_\theta x)$
and $g_n(x)$ by $f_n(x)=f_n(g_n(x))$, so that
  $g_n\big( (\mle-b(\theta))/\sigma_\theta \big)
  =(\theta^*-b(\theta))/\sigma_\theta$.
One finds 
  $f_n^{(2)}(0)=2\eta'(b(\theta))\sigma_\theta^2$
and 
  $f_n^{(k)}(0)=2\eta^{(k-1)}(b(\theta))\sigma_\theta^k$
for $k\geq 3$.
It is easy to show that $\eta^{(k)}(x)=O(n)$ for any integer $k$, cfr. \eqref{eq:eta_prime}, so the hypothesis of Lemma 1 are satisfied. 
Hence an application of Lemma \ref{lemma:1} yields
\begin{multline}\label{eq:Lemma1} 
  {\theta^*-b(\theta)\over \sigma_\theta}
  =-{\mle-b(\theta)\over \sigma_\theta}
  -{1\over 3}{\eta'(b(\theta))\sigma_\theta\over \eta''(b(\theta))}
  \bigg( {\mle-b(\theta)\over \sigma_\theta} \bigg)^2\\
  -\bigg({1\over 3}
  {\eta'(b(\theta))\sigma_\theta\over \eta''(b(\theta))}\bigg)^2
  \bigg( {\mle-b(\theta)\over \sigma_\theta} \bigg)^3
  +O(n^{-3/2})
\end{multline}
for $[\mle-b(\theta)]/\sigma_\theta=O_p(1)$. Based on \eqref{eq:median}, we have
  $[\theta^{**}-b(\theta)]/\sigma_\theta
  =(\theta^{**}-\theta)/\sigma_\theta
  +\rho_3/6n^{1/2}+O(n^{-3/2})$ 
and
  $(\mle-b(\theta))/\sigma_\theta
  =U
  +\rho_3/6n^{1/2}+O(n^{-3/2})$. 
Moreover, using a simple Taylor expansion, \eqref{eq:eta_prime} and $\eta^{(k)}(x)=O(n)$ for any $k$, it can be shown that
  $\eta'(b(\theta))/\eta''(b(\theta))
  =-\rho_3/n^{1/2}\sigma_\theta+O(n^{-1})$. 
Hence, we can reduce \eqref{eq:Lemma1} to
\begin{align}
  {\theta^*-\theta\over \sigma_\theta}
  &=-U-{2\rho_3\over 6n^{1/2}}
  +{1\over 3}{\rho_3\over n^{1/2}}
  U^2
  +{1\over 3}{\rho_3\over n^{1/2}}\bigg(2U{\rho_3\over6n^{1/2}}\bigg)
  -\bigg({1\over 3}
  {\rho_3\over n^{1/2}}\bigg)^2
  U^3
  +O(n^{-3/2})\notag\\
  &=-U+{\rho_3\over 3n^{1/2}}(U^2-1)
  -\bigg({1\over 3}
  {\rho_3\over n^{1/2}}\bigg)^2
  (U^3-U)
  +O(n^{-3/2})\label{eq:th.star}
\end{align}
for $U=O_p(1)$. As for the asymptotic expansion of $\theta^{**}$ in \eqref{eq:5}, let $h_n(x)$ satisfy $K_\theta(h_n(x))=1-K_\theta(x)$ so that $(\theta^{**}-\theta)/\sigma_\theta=h_n(U)$. Using Cornish-Fisher expansion \eqref{eq:corn_fish} one finds that, for any $\alpha\in(0,1)$,
  $$K^{-1}(\alpha)=-K^{-1}(1-\alpha)+\rho_3
  (z_{1-\alpha}^2-1)/3n^{1/2}+O(n^{-3/2})$$
so that
  $h_n(x)=-x+\rho_3
  [\Phi^{-1}(1-K_\theta(x))^2-1]/3n^{1/2}+O(n^{-3/2})$  
for $x=O(1)$. Hence,
\begin{equation}\label{eq:th.2star}
  {\theta^{**}-\theta\over \sigma_\theta}
  =-U+{\rho_3\over 3n^{1/2}}
  [\Phi^{-1}(1-K_\theta(U))^2-1]+O(n^{-3/2})
\end{equation}  
for $U=O_p(1)$. Next, use the Edgeworth expansion for $K_\theta(U)$ up to the first term, i.e.
  $1-K_\theta(U)
  =\Phi(-U)+\phi(-U)\rho_3(U^2-1)/6n^{1/2}+O(n^{-1})$
and a Taylor expansion of $\Phi^{-1}(x+\Delta x)$ at $x=\Phi(-U)$ for $\Delta x=1-K_\theta(U)-\Phi(-U)$ to get
  $$\Phi^{-1}(1-K_\theta(U))
  =-U+\rho_3(U^2-1)/6n^{1/2}+O(n^{-1}).$$
Substitution into \eqref{eq:th.2star} leads to
an asymptotic expansion of $(\theta^{**}-\theta)/\sigma_\theta$ which corresponds to \eqref{eq:th.star}. 
Hence, \eqref{eq:5} follows and the proof is complete.
\end{proof of}
%
%
\begin{proof of}{Theorem \ref{theorem:2}}
In order to prove \eqref{eq:asy_r}, we proceed by deriving two asymptotic expansions for $r^*(\theta)$ and $r(b(\theta))$ and by showing that they coincide up to the required order. As for $r^*(\theta)$, we resort to equation (2.4)--(2.6) in \citet{Bar:90}. After some algebra and further expansion,
  $$r(\theta)=u(\theta)\left(
  1+\frac{1}{6}(\cancel{\ell}_3+3\cancel{\ell}_{2;1})
  (\mle-\theta)\cancel{j}^{-1}+O(n^{-1})\right)$$
so that
\begin{equation}\label{eq:bar90}
  r^*(\theta)
  =r(\theta)-\frac{1}{6}\frac{1}{r(\theta)}
  (\cancel{\ell}_3+3\cancel{\ell}_{2;1})
  (\mle-\theta)\cancel{j}^{-1}+O(n^{-1})
\end{equation}
where we have also used $\log(1+x)=x+O(x^2)$ for $|x|$ small. As for $r(b(\theta))$, a Taylor expansion around $\theta$ gives
\begin{equation}\label{eq:taylor1}
  r(b(\theta))
  =r(\theta)-\frac{1}{r(\theta)}\ell_1(\theta)
  (b(\theta)-\theta)+R_n,
\end{equation}
with $R_n$ denoting the remainder. In the one-parameter exponential family, borrowing the notation from the proof of Theorem \ref{theorem:2}, we have
\begin{equation}\label{eq:taylor2}  
  \ell_1(\theta)=\eta'(\theta)(\mle-\theta)
  =\cancel{j}(\mle-\theta)
\end{equation}  
since $\eta'(\theta)=1/\psi^{(2)}(\eta)=\sigma_\theta^{-2}=\cancel{j}$. Moreover, \eqref{eq:median} in the proof of Theorem \ref{theorem:2} implies that
\begin{equation}\label{eq:taylor3}
  b(\theta)-\theta
  =-\frac{1}{6\sqrt n}
  \rho_3\sigma_\theta+O(n^{-2})
  =\frac{1}{6}\cancel{j}^{-2}
  (\cancel{\ell}_3+3\cancel{\ell}_{2;1})+O(n^{-2})
\end{equation}
since $\cancel{\ell}_3=-2\eta''(\theta)$, $\cancel{\ell}_{2;1}=\eta''(\theta)$ and $\eta''(\theta)=-\psi^{(3)}(\eta)/\psi^{(2)}(\eta)^3=\sigma_\theta^{-2}=\cancel{j}=-\rho_3/(\sigma_\theta^3\sqrt n)$. Inserting \eqref{eq:taylor2} and \eqref{eq:taylor3} into \eqref{eq:taylor1} we obtain the same expansion in \eqref{eq:bar90} provided that the remainder $R_n$ is $O(n^{-1})$. This can be shown by using $\ell_1(\theta)=O(n^{1/2})$, $\ell_k(\theta)=O(n)$, $k\geq 2$ and $r(\theta)=O(1)$ in the normal deviation range, together with $b(\theta)-\theta=O(n^{-1})$. Hence \eqref{eq:asy_r} follows.
\end{proof of}
%
%
\begin{lemma}\label{lemma:1}
Let $\{f_n(x)\}_{n\geq 1}$ be a sequence of infinitely differentiable convex functions with minimum at $x=0$ and $f_n(0)=0$, and let $g_n(x)$ be defined by $f_n(x)=f_n(g_n(x))$. For $b_{n,k}=2f_n^{(k)}(0)/k!f_n^{(2)}(0)$, assume that, as $n\to\infty$, 
  $b_{n,k}=O(b_{n,k-1}n^{-1/2})$
for any $k\geq 3$. 
Then, $g_n(x)$ 
admits asymptotic expansion
  $$g_n(x)=-x-\sum_{k\geq 2}a_{n,k}x^k,$$
where $a_{n,2}=b_{n,3}$, $a_{n,3}=b_{n,3}^2$ and
\begin{equation}\label{eq:system}
 a_{n,k}=
 \left\{\begin{array}{ll}
 O(b_{n,k+1})
 &(k\mbox{ even}),\\
 O(b_{n,3}b_{n,k})
 &(k\mbox{ odd}).
 \end{array}\right.
\end{equation} 
\end{lemma}
%
%
\begin{proof}
We omit the subscript $n$ for ease of notation. Taylor expansion of $f$ at $x=0$ gives
 $f(x)=(1/2)f^{(2)}(0)(x^2+
 b_3x^3+\ldots+b_{k}x^k+\ldots)$.
Substitute
 $g(x)=-x-a_2x^2-\ldots,$
into $f(x)=f_n(g_n(x))$ and equate coefficients of successive order to obtain 
  $$\left\{\begin{array}{ll}
  b_3=&(a_1a_2+a_2a_1)-b_3\\
  b_4=&(a_1a_3+a_2a_2+a_3a_1)-b_3(a_1a_1a_2+a_1a_2a_1+a_2a_1a_1)+b_4\\
  \ldots&\ldots\\
  b_k=&\sum_{i_1+i_2=k}a_{i_1}a_{i_2}
  -b_3\sum_{i_1+i_2+i_3=k}a_{i_1}a_{i_2}a_{i_3}
  +b_4\sum_{i_1+\ldots+i_4=k}a_{i_1}\cdots a_{i_4}\\
  &+\ldots+
  (-1)^{k-1}b_{k-1}
  \sum_{i_1+\ldots+i_{k-1}=k}a_{i_1}\cdots a_{i_{k-1}}
  +(-1)^{k}b_k
  \end{array}\right.$$
where the $i_j$'s are positive integers and we set $a_1=1$ for notational convenience. Rearranging terms, the first $4$ equations are
  $$\left\{\begin{array}{ll}
  b_3=&-b_3 + 2a_1a_2\\
  b_4=&+b_4-b_3(3a_1^2a_2) + 2a_1a_3+a_2^2\\
  b_5=&-b_4+b_4(4a_1^3a_2)
  -b_3(3a_1^2a_3+3a_1a_2^2) + 2a_1a_4+2a_2a_3\\
  b_6=&+b_6-b_5(5a_1^4a_2)+b_4(4a_1^3a_3+6a_1^2a_2^2)\\
  & -b_3(3a_1^2a_4+6a_1a_2a_3+a_2^3)
  + 2a_1a_5+2a_2a_4+a_3^2
  \end{array}\right.$$
A similar expression for $b_k$ can be given by means of multinomial coefficients. Now substitute back $a_1=1$,
and solve for $a_2,a_3,a_4,a_5$ to get $a_2=b_3$, $a_3=b_3^2$ and 
  $$\left\{\begin{array}{lll}
  a_4=&b_5-2b_3b_4+2b_3^3&=O(b_5)
  \\
  a_5=&3b_3b_5-6b_3^2b_4+4b_3^4&=O(b_3b_5)
  \end{array}\right.$$
where the order of asymptotics of $a_4$ and $a_5$ are determined by the hypothesis $b_k=O(b_{k-1}n^{-1/2})$. An argument by induction leads to \eqref{eq:system}.
\end{proof}


\section*{Acknowledgements}
The authors are grateful to two reviewers for comments that have helped to improve the paper substantially. Special thanks are also due to Igor Pr\"unster and to Mattia Ciollaro for comments on an earlier version of this work. P. De Blasi was supported by the European Research Council (ERC) through StG \comillas{N-BNP} 306406.


\end{document}